\documentclass[14pt,a4paper]{article}

\usepackage{lineno}
\usepackage{mathrsfs}
\usepackage{amscd}
\usepackage{hyperref}
\hypersetup{colorlinks=true,citecolor=green,linkcolor=blue,urlcolor=blue}
\usepackage{tipa}
\usepackage{amssymb}
\usepackage{amsfonts}
\usepackage{stmaryrd}
\usepackage{amssymb}

\usepackage{CJK}  
\usepackage{indentfirst} 

\usepackage{latexsym}   
\usepackage{bm}         

\usepackage{pifont}
\usepackage{bm}
\usepackage{amsmath,amssymb,amsfonts}
\usepackage{indentfirst}

\usepackage{xcolor}
\usepackage{cases}
\usepackage{wasysym}
\usepackage{amsthm}

\newtheorem{theorem}{Theorem}[section]
\newtheorem{lemma}[theorem]{Lemma}
\newtheorem{proposition}[theorem]{Proposition}

\newtheorem{remark}[theorem]{Remark}
\newtheorem{corollary}[theorem]{Corollary}
\newtheorem{definition}[theorem]{Definition}
\CJKtilde   

\begin{document}

\title{\Large \textbf{A note on the universal supersingular quotients of $U(2,1)$}}
\date{}
\author{Peng Xu}

\maketitle

\begin{abstract}
Let $G$ be the unramified unitary group $U(2, 1)(E/F)$ defined over a non-archimedean local field $F$ of residue characteristic $p\neq 2$. In this note, we prove the universal supersingular quotients of $G$ are not irreducible in general.
\end{abstract}

\section{Introduction}\label{sec: intro}

Let $G$ be the unramified unitary group $U(2, 1)(E/F)$ defined over a non-archimedean locally compact field $F$ of residue characteristic $p\neq 2$, and let $K$ be a maximal compact open subgroup of $G$. For an irreducible smooth $\overline{\mathbf{F}}_p$-representation $\pi$ of $G$, and a weight $\sigma$ of $K$ contained in $\pi$, the space
\begin{center}
$\text{Hom}_G (\text{ind}^G _K \sigma, \pi)$
\end{center}
is a right module over the spherical Hecke algebra $\mathcal{H}(K, \sigma): =\text{End}_G (\text{ind}^G _K \sigma)$. In \cite[Theorem 1.1]{X2018}, we proved the above space admits eigenvectors for $\mathcal{H}(K, \sigma)$. As the algebra $\mathcal{H}(K, \sigma)$ is isomorphic to $\overline{\mathbf{F}}_p [T_\sigma]$ for some $T_\sigma \in \mathcal{H}(K, \sigma)$, the representation $\pi$ is a quotient of $\text{ind}^G _K \sigma /(T_\sigma -\lambda)$ for some $\lambda \in \overline{\mathbf{F}}_p$. In the case of $\lambda=0$ the representation $\text{ind}^G _K \sigma /(T_\sigma)$ is usually called the universal supersingular quotient of $G$. In this note, we prove the following result for such representations.

\begin{theorem}(Corollary \ref{reducibility of supersingular quotient}, Corollary \ref{reducibility of supersingular quotient: the case K=K_1 and q=p})\label{main theorem: intro}
Suppose $K$ is special but non-hyperspecial, and the size of the residue field of $F$ is $p$. For any weight $\sigma$ of $K$, the universal supersingular quotient $\textnormal{ind}^G _K \sigma /(T_\sigma)$ is not irreducible.
\end{theorem}

 We prove Theorem \ref{main theorem: intro} by finding a pro-$p$-Iwahori invariant function from the maximal compact induction and showing that its non-zero image in the universal quotient does not generate the representation. We expect the theorem still holds without the restriction (on the group $K$ and and the size of the residue field) in the statement; indeed they are only used in Corollary \ref{reducibility of supersingular quotient: the case K=K_1 and q=p}.

To the author's knowledge, the universal supersingular quotients have only been extensively studied for the group $GL_2 (F)$ (\cite{Mor12},\cite{Mor13},\cite{Sch14},\cite{Yot19}). Besides the definition (\cite{AHHV17a}), very little is known about them in the general case. It is interesting to see whether our approach works for other groups, and our guess is that the function we found (to make the idea work) is very specific to the group itself: for example some quick computation indicates that the analogue of such functions does not exist for the group $GL_2 (F)$.

\section{Notations and Preliminary}\label{sec: notations}

\subsection{Notations}

Let $F$ be a non-archimedean local field of odd residue characteristic $p$, with ring of integers $\mathfrak{o}_F$ and maximal ideal $\mathfrak{p}_F$, and let $k_F$ be its residue field of cardinality $q=p^f$. Fix a separable closure $F_s$ of $F$. Let $E$ be the unramified quadratic extension of $F$ in $F_s$. We use similar notations $\mathfrak{o}_E$, $\mathfrak{p}_E$, $k_E$ for analogous objects of $E$. Fix a uniformizer $\varpi_{E}$ of $E$.

Equip $E^3$ with the non-degenerate Hermitian form h:
\begin{center}
 $\text{h}:~E^3 \times E^3 \rightarrow E$,
$(v_{1}, v_{2}) \mapsto ~v_{1}^{\text{T}}\beta \overline{v_2}, v_{1}, v_{2}\in E^3$.
\end{center}
Here, $-$ denotes the non-trivial Galois conjugation on $E/F$, and
$\beta$ is the matrix
\[ \begin{matrix}\begin{pmatrix} 0  & 0 & 1  \\ 0  & 1 & 0\\
1 & 0 & 0
\end{pmatrix}
\end{matrix}. \]
The unitary group $G$ is defined as:

\begin{center}
$G=\{g\in \text{GL}(3, E)\mid \text{h}(gv_1, gv_2)= \text{h}(v_1, v_2), \forall v_1, v_2\in E^3\}.$
\end{center}

Let $B=HN$ (resp, $B'= HN'$) be the subgroup of upper (resp, lower) triangular matrices of $G$, where $N$ (resp, $N'$) is the unipotent radical of $B$ (resp, $B'$) and $H$ is the diagonal subgroup of $G$. Denote an element of the following form in $N$ and $N'$ by $n(x, y)$ and $n'(x, y)$ respectively:
\begin{center}
$\begin{pmatrix}  1 & x & y  \\ 0 & 1 & -\bar{x}\\
0 & 0 & 1
\end{pmatrix}$, ~
$\begin{pmatrix}  1 & 0 & 0   \\ x & 1 & 0\\
y & -\bar{x} & 1
\end{pmatrix}$,
\end{center}
where $(x, y)\in E^2$ satisfies $x\bar{x}+ y+ \bar{y}=0$. For any $k\in \mathbb{Z}$, denote by $N_k$ (resp, $N'_k$) the subgroup of $N$ (resp, $N'$) consisting of all $n(x, y)$ (resp, $n'(x, y)$) with $y\in \mathfrak{p}^{k}_E$. For $x\in E^\times$, denote by $h(x)$ an element in $H$ of the following form:
\begin{center}
$\begin{pmatrix}  x & 0 & 0  \\ 0 & -\bar{x}x^{-1} & 0\\
0 & 0 & \bar{x}^{-1}
\end{pmatrix}.$
\end{center}

We record the following useful identity in $G$: for $y\neq 0$,
\begin{equation}\label{useful identity}
\beta n(x, y)= n(\bar{y}^{-1}x, y^{-1})\cdot h(\bar{y}^{-1})\cdot n'(-\bar{y}^{-1}\bar{x}, y^{-1}).
\end{equation}

\medskip
Up to conjugacy, the group $G$ has two maximal compact open subgroups $K_0$ and $K_1$, given by:
\begin{center}
$K_0= \begin{pmatrix}  \mathfrak{o}_E & \mathfrak{o}_E & \mathfrak{o}_E  \\ \mathfrak{o}_E  & \mathfrak{o}_E & \mathfrak{o}_E\\
\mathfrak{o}_E & \mathfrak{o}_E & \mathfrak{o}_E
\end{pmatrix}\cap G, ~K_1= \begin{pmatrix}  \mathfrak{o}_E & \mathfrak{o}_E & \mathfrak{p}^{-1}_E  \\ \mathfrak{p}_E  & \mathfrak{o}_E & \mathfrak{o}_E\\
\mathfrak{p}_E & \mathfrak{p}_E & \mathfrak{o}_E
\end{pmatrix}\cap G$.
\end{center}

Let $\alpha$ be the following diagonal matrix in $G$:
\[ \begin{matrix}\begin{pmatrix} \varpi_{E}^{-1}  & 0 & 0  \\ 0  & 1 & 0\\
0 & 0 & \varpi_{E}
\end{pmatrix}
\end{matrix} ,\]
and put $\beta'=\beta \alpha^{-1}$. Note that $\beta\in K_0$ and $\beta'\in K_1$. We use $\beta_K$ to denote the unique element in $K\cap \{\beta, \beta'\}$.

 Let $K\in\{K_0, K_1 \}$, and $K^1$ be the maximal normal pro-$p$ subgroup of $K$. The finite group $\Gamma_K= K/K^1$ may be identified with the $k_F$-points of an algebraic group defined over $k_F$, denoted also by $\Gamma_K$: when $K$ is $K_0$, $\Gamma_K$ is $U(2, 1)(k_E /k_F)$, and when $K$ is $K_1$, $\Gamma_K$ is $U(1, 1)\times U(1)(k_E /k_F)$. Let $\mathbb{B}$ be the upper triangular subgroup of $\Gamma_K$, and $\mathbb{U}$ be its unipotent radical. The Iwahori subgroup $I_K$ and pro-$p$ Iwahori subgroup $I_{1,K}$ in $K$ are the preimages of $\mathbb{B}$ and $\mathbb{U}$ in $K$.

Put $H_0= H\cap I_K$, and $H_1= H\cap I_{1, K}$. As $H_0 /H_1 \cong I_K / I_{1, K}$, we will identify the characters of these groups. For a character $\chi$ of $H_0$, i.e., a character of $H_0 /H_1$, denote by $\chi^s$ the character given by $\chi^s (h):= \chi (\beta_K h \beta^{-1}_K)$.

\smallskip

Denote by $n_K$ and $m_K$ the unique integers such that $N\cap I_{1, K}= N_{n_K}$ and $N'\cap I_{1, K}=N'_{m_K}$. Note that the coset spaces $N_{n_K}/N_{n_K +1}$ and $N'_{m_K}/N'_{m_K +1}$ are finite groups of order respectively $q^{t_K}$ and $q^{4-t_K}$, where $t_K= 3$ or $1$, depending on $K$ is $K_0$ or $K_1$.

The following group:
\begin{center}
$L_{q^3} := \{(x, t)\in k^2 _E \mid x\bar{x} + t +\bar{t}=0\}$,
\end{center}
and its central subgroup:
\begin{center}
$L_q := \{(0, t)\in k^2 _E\mid t +\bar{t}=0\}$.
\end{center}
will be used later. Here, the group structure of $L_{q^3}$ is given by
\begin{center}
$(x, t) \cdot (x', t') := (x+x', t+ t'- x' \bar{x}).$
\end{center}
We note that $\mid L_{q^3} \mid=q^3$ and that $\mid L_q \mid= q$.

We may identify these groups naturally:
\begin{center}
$L_{n_K}: N_{n_K}/N_{n_K +1} \cong L_{q^{t_K}}$

$n(x, \varpi^{n_K} _E t) \mapsto (x\varpi^{-n_K}_E , t) (\text{mod}~\mathfrak{p}_E)$
\end{center}

\begin{center}
$L_{m_K}: N'_{m_K}/N'_{m_K +1}\cong L_{q^{4-t_K}}$

$n'(\varpi_E x, \varpi^{m_K} _E t) \mapsto (x\varpi^{2-m_K}_E, t)(\text{mod}~\mathfrak{p}_E)$
\end{center}
Here, the elements $x$ and $t$ on the left hand side lie in $\mathfrak{o}_E$.

We fix a non-zero element $\mathfrak{t}\in \mathfrak{o}^\times _E$ with trace zero.

All the representations of $G$ and its subgroups considered in this paper are smooth over $\overline{\mathbf{F}}_p$.

\subsection{The spherical Hecke algebra $\mathcal{H}(K, \sigma)$}\label{subsec: spherical hecke algebra}

Let $K$ be a maximal compact open subgroup of $G$, and $(\sigma, W)$ be an irreducible smooth representation of $K$. As $K^1$ is pro-$p$ and normal, $\sigma$ factors through the finite group $\Gamma_K= K/K^1$, i.e., $\sigma$ is the inflation of an irreducible representation of $\Gamma_K$. Conversely, any irreducible representation of $\Gamma_K$ inflates to an irreducible smooth representation of $K$. We may therefore identify irreducible smooth representations of $K$ with irreducible representations of $\Gamma_K$, and we shall call them \emph{weights} of $K$ or $\Gamma_K$ from now on. It is known that $\sigma^{I_{1,K}}$ and $\sigma_{I'_{1,K}}$ are both one-dimensional, and that the natural composition map $\sigma^{I_{1,K}}\hookrightarrow \sigma \twoheadrightarrow \sigma_{I'_{1,K}}$ is an isomorphism of vector spaces (\cite[Theorem 6.12]{C-E2004}). Denote by $j_\sigma$ the inverse of the map aforementioned. For $v\in \sigma^{I_{1,K}}$, we have $j_\sigma (\bar{v})= v$, where $\bar{v}$ is the image of $v$ in $\sigma_{I'_{1,K}}$. By composition, we view $j_\sigma$ as a map in $\text{End}_{\overline{\mathbf{F}}_p} (\sigma)$.

 Let $\text{ind}_K ^G \sigma$ be the smooth representation of $G$ compactly induced from $\sigma$, i.e., the representation of $G$ with underlying space $S(G, \sigma)$
\begin{center}
$S(G, \sigma)=\{f: G\rightarrow W\mid  f(kg)=\sigma (k)\cdot f(g), \forall~k\in K, g\in G,~ \text{smooth~with~compact~support}\}$
\end{center}
and $G$ acting by right translation. In this paper, we will sometimes call $\text{ind}_K ^{G}\sigma$ a maximal compact induction. As usual, denote by $[g, v]$ the function in $S(G, \sigma)$, supported on $Kg^{-1}$ and having value $v\in W$ at $g^{-1}$.  An element $g'\in G$ acts on the function $[g, v]$ by $g'\cdot[g, v]= [g'g, v]$, and we have $[gk, v]= [g, \sigma(k)v]$ for $k\in K$.

The spherical Hecke algebra $\mathcal{H}(K, \sigma)$ is defined as $\text{End}_G (\text{ind}^G _K \sigma)$, and it is isomorphic to $\overline{\mathbf{F}}_p [T]$(\cite[Corollary 1.3]{Her2011a}), for certain $T \in \mathcal{H}(K, \sigma)$. We explain below the Hecke operator $T$ in detail. By \cite[Proposition 5]{B-L94}, the algebra $\mathcal{H}(K, \sigma)$ is isomorphic to the convolution algebra $\mathcal{H}_K (\sigma)$:
\begin{center}
$\mathcal{H}_K (\sigma)= \{\varphi: G\rightarrow \text{End}_{\overline{\mathbf{F}}_p}(\sigma) \mid \varphi(kgk')=\sigma(k)\varphi(g)\sigma(k'), \forall~k, k'\in K, g\in G, ~\text{smooth~with~compact~support}\}$
\end{center}

Let $\varphi$ be the function in $\mathcal{H}_K (\sigma)$, supported on $K\alpha K$ and satisfying $\varphi (\alpha)= j_\sigma$. Let $T$ be the unique element in $\mathcal{H}(K, \sigma)$ which corresponds to the function $\varphi$, via the isomorphism aforementioned between $\mathcal{H}_K (\sigma)$ and $\mathcal{H}(K, \sigma)$. We refer the reader to \cite[(4)]{X2019} for the following formula of $T$: for a $v \in \sigma$, we have

\begin{equation}\label{explicit form for T}
T [Id, v]=\sum_{u\in N_{n_K}/N_{n_K+ 2}} [u \alpha^{-1}, j_{\sigma}\sigma (u^{-1})v]+ \sum_{u\in N_{n_K +1}/N_{n_K+ 2}}~[\beta_K u\alpha^{-1}, j_{\sigma}\sigma(\beta_K) v]
\end{equation}

\subsection{The image of $(\textnormal{ind}_K ^G \sigma)^{I_{1, K}}$  under the Hecke operator $T$}\label{subsection: image of I_1 in the quotient}

We fix a non-zero vector $v_0\in \sigma^{I_{1, K}}$. Let $f_n$ be the function in $(\textnormal{ind}_K ^G \sigma)^{I_{1, K}}$, supported on $K \alpha^{-n} I_{1, K}$, such that
\begin{center}
$f_n (\alpha^{-n})=  \begin{cases}
\beta_K\cdot v_0, ~~~~~~~n>0,\\
v_0 ~~~~~n\leq 0.
\end{cases}$
\end{center}
Then, we have (\cite[Lemma 3.5]{X2019})

\begin{lemma}
The set of functions $\{f_n \mid n\in \mathbb{Z}\}$ consists of a basis of the $I_{1, K}$-invariants of the maximal compact induction $\textnormal{ind}^{G} _K\sigma$.
\end{lemma}

\begin{proposition}\label{hecke operator formula}

We have:

$(1)$~~~$T\cdot f_0= f_{-1} + \lambda_{\beta_K, \sigma}\cdot f_1$.

$(2)$~~~For $n\neq 0$, $T \cdot f_n= cf_n +f_{n+\delta(n)}$, where $c$ is a constant (depending on $\sigma$) and $\delta(n)$ is either $1$ or $-1$, depending on $n> 0$ or $< 0$.
\end{proposition}

\begin{proof}
$(1)$ is \cite[Proposition 3.6]{X2019}), and $(2)$ is \cite[Corollary 3.11]{X2019}. The value of $c$ is not recorded explicitly in \emph{loc.cit}, but by the same argument we can check that it is zero if $\text{dim}_{\overline{\mathbf{F}}_p} \sigma > 1$; when $\sigma$ is a character, it is equal to $\sum_{(x, t)\in L^\times _{q^{4-t_K}}}\chi_\sigma (h(t))$.
\end{proof}

\smallskip
\emph{We will occasionally write $f_{n,\sigma}$ for $f_n$ to indicate that the function is defined with respect to a specific weight $\sigma$.}

\subsection{The supersingular universal quotient of $G$}

Let $\pi$ be an irreducible smooth representation of $G$, and let $\sigma$ be a weight of $K$ contained in $\pi$. By composition, the space
\begin{center}
$\text{Hom}_G (\text{ind}^G _K \sigma, \pi)$
\end{center}
is a right module over the spherical Hecke algebra $\mathcal{H}(K, \sigma): =\text{End}_G (\text{ind}^G _K \sigma)$. The main result of \cite{X2018} is:
\begin{theorem}
The space $\text{Hom}_G (\textnormal{ind}^G _K \sigma, \pi)$ admits eigenvectors for the spherical Hecke algebra $\mathcal{H}(K, \sigma)$.
\end{theorem}

We modify the Hecke operator $T$ slightly as follows:

\begin{center}
$T_\sigma =\begin{cases}
T,~~~~~~~~~~~~~~\textnormal{dim}\sigma > 1\\
T, ~~~~~~~~~~~~~~\textnormal{dim}\sigma = 1, \chi_\sigma= \chi^s _\sigma~\chi_\sigma \neq \chi\circ\text{det};\\
T+1, ~~~~~~~\textnormal{dim}\sigma = 1, \chi_\sigma= \chi\circ\text{det}.
\end{cases}$
\end{center}

As the algebra $\mathcal{H}(K, \sigma)$ is isomorphic to $\overline{\mathbf{F}}_p [T_\sigma]$, the representation $\pi$ is isomorphic to a quotient of $\text{ind}^G _K \sigma /(T_\sigma -\lambda)$, for some $\lambda \in \overline{\mathbf{F}}_p$.  In the case $\lambda= 0$, we encounter the so-called universal supersingular quotient of $G$, as the main theme we will deal with in the present paper.

\section{The space $\text{Hom}_G (\text{ind}^G _K \sigma, \text{ind}^G _K \sigma')$}

Let $\sigma$ and $\sigma'$ be two weights of $K$, and denote by $L(\sigma, \sigma')$ the space
\begin{center}
$\text{Hom}_G (\text{ind}^G _K \sigma, \text{ind}^G _K \sigma')$.
\end{center}

\emph{The following two lemmas should be known in quite generality, even our arguments here seems different.}
\begin{lemma}
If $\chi_{\sigma'} \notin \{\chi_\sigma, \chi^s _\sigma \}$, then the space $L(\sigma, \sigma')$ is zero.
\end{lemma}

\begin{proof}
If the space $L(\sigma, \sigma')$ is not zero, any non-zero map in the space sends $f_{0, \sigma}$ to a function of the form $\sum_k c_k f_{k, \sigma'}$. By considering the action of $I_K$ on both sides, we get a contradiction under the assumption of the Lemma.
\end{proof}

\begin{lemma}
Let $\sigma$ be a weight such that $\chi_\sigma \neq \chi^s _\sigma$. We have
\begin{center}
$L(\sigma, \sigma^s)=0$
\end{center}
\end{lemma}
\begin{proof}
As above, any non-zero $G$-map sends $f_{0, \sigma}$ to a function of the form $\sum_k c_k f_{k, \sigma^s }$. By considering the action of $I_K$ on both sides, we see the sum will only have terms in the positive part, that is of the form $\sum_{k\geq1} c_k f_{k, \sigma^s }$.

Now we apply the operator $S_K$ to $f_{0, \sigma}$. By our assumption on $\sigma$ and $(2)$ of Proposition \ref{the image of I_1 under S_K and S_-}, we get $S_K f_{0,\sigma} =0$. However, by $(1)$ of the same Proposition, we have
\begin{center}
$S_K (\sum_{k\geq1} c_k f_{k, \sigma^s })= \sum_{k\geq1} c_k f_{-k, \sigma^s }\neq 0$
\end{center}
As any $G$-map respects the action of $S_K$, that is a contradiction.
\end{proof}

\begin{remark}
One may rephrase the Lemma as follows. By Frobenius reciprocity, for a weight $\sigma$ satisfying $\chi_\sigma \neq \chi^s _\sigma$  the maximal compact induction $\textnormal{ind}^G _K \sigma$ does not contain the weight $\sigma^s$.
\end{remark}

The Lemmas above show that the remaining interesting case is $\{\sigma, \sigma'\}= \{1, st\}$, which we deal with in the following part.

\begin{proposition}\label{change of weight}
We have

$(1)$.~ For a non-zero function $f \in \langle f_{-n} \mid n\geq 1\rangle_{\overline{\mathbf{F}}_p}$, there exists a non-zero $G$-map $S_f$ from $\textnormal{ind}^G _K st$ to $\textnormal{ind}^G _K 1$, characterized by $S_f (f_0) =f$.

$(2)$.~The map $f \mapsto S_f$ from $(1)$ gives an isomorphism of vector spaces:
\begin{center}
$\langle f_{-n} \mid n\geq 1\rangle_{\overline{\mathbf{F}}_p} \cong \textnormal{Hom}_G (\textnormal{ind}^G _K st, \textnormal{ind}^G _K 1)$.
\end{center}
\end{proposition}

\begin{proof}
The argument here is motivated by that of \cite[Lemma 1.5.5]{Ki09}. Suppose we are given a non-zero $f \in \langle f_{-n}, n\geq 1\rangle_{\overline{\mathbf{F}}_p}$. The Iwahori group $I_K$ acts trivially on $f$ (as the character $\chi_\sigma =1$), so the representation $\textnormal{ind}^G _K 1$ contains the trivial character of $I_K$. By Frobenius reciprocity, we get a $K$-map $\varphi_f$ from the finite principal series $\text{Ind}^K _{I_K} 1$ to $\textnormal{ind}^G _K 1$, sending $1_{I_K}$ to $f$. We claim that $\varphi_f$ kills the constant function in $\text{Ind}^K _{I_K} 1$. In other words, we have:
\begin{center}
$\sum_{k\in K/I_K} k\cdot f =0.$
\end{center}
Note that the following gives a set of representatives for $K/I_K$:
\begin{center}
$\{Id\} \cup \{[u]\beta_K\mid u\in N_{n_K}/N_{n_K +1}\}$.
\end{center}
So the above sum $\sum_{k\in K/I_K} k\cdot f$ reads as
\begin{center}
$f + \sum_{u\in N_{n_K}/N_{n_K +1}} u\beta_K \cdot f$.
\end{center}
In our earlier notation, the second part is simply $S_K \cdot f$. As $f\in \langle f_{-n}, n\geq 1\rangle$,  we see from $(2)$ of Proposition \ref{the image of I_1 under S_K and S_-} that
\begin{center}
$S_K f= -f$,
\end{center}
whence the claim.

In summary, the $K$-map $\varphi_f$ factors through $\text{Ind}^K _{I_K} 1/ 1$, i.e., the Steinberg weight $st$ of $K$. By Frobenius reciprocity again, we get a $G$-map from $\textnormal{ind}^G _K st$ to $\textnormal{ind}^G _K 1$, sending $f_0$ to $f$, that is the map $S_f$ as required.

\smallskip

By revising the argument above and using Proposition \ref{the image of I_1 under S_K and S_-}, we see the map $S_f$ is indeed an isomorphism and we get $(2)$.
\end{proof}

The following is similar but even easier.
\begin{proposition}\label{basis from the trivial weight to steinberg weight}
We have an isomorphism of spaces
\begin{center}
$\textnormal{Hom}_G (\textnormal{ind}^G _K 1, \textnormal{ind}^G _K st)\cong \langle f_n + f_{-n} \mid n\geq 1\rangle$
\end{center}
\end{proposition}
\begin{proof}
Frobenius reciprocity says the first space is isomorphic to the space of $K$-invariants of $\textnormal{ind}^G _K st$, and the set of functions $\{ f_n + f_{-n} \mid n\geq 1 \}$ gives a basis of the latter space.
\end{proof}

\section{The universal quotient $\text{ind}^G _K \sigma/ (T_\sigma)$ is not irreducible}

For a weight $\sigma$ of $K$, the representation $\text{ind}^G _K \sigma/ (T_\sigma)$ is usually called the universal supersingular quotient of $G$. This is because any supersingular representation of $G$ is a quotient of $\text{ind}^G _K \sigma/ (T_\sigma)$ for some $\sigma$. We prove in this part that $\text{ind}^G _K \sigma/ (T_\sigma)$ is not irreducible in general.

\subsection{The degenerate case}

We begin with some simple but important observation.

\begin{lemma}\label{f_1 generate the whole rep}
Assume $\textnormal{dim}~\sigma > 1$. If $\overline{f_1}$ generates $\textnormal{ind}^G _K \sigma /(T)$, then $f_1$  generates $\textnormal{ind}^G _K \sigma$.
\end{lemma}

\begin{proof}
Suppose $\overline{f_1}$ generates $\textnormal{ind}^G _K \sigma /(T)$. Then we have
\begin{center}
$f_0 \in \langle f_1\rangle_G + (T)= \langle f_1\rangle_G + \langle T f_0\rangle_G=\langle f_1\rangle_G + \langle f_{-1}\rangle_G=\langle f_1\rangle_G + \langle S_K f_1\rangle_G = \langle f_1 \rangle_G$.
\end{center}
Here, we have used Proposition \ref{hecke operator formula} and Proposition \ref{the image of I_1 under S_K and S_-} for the second and third equality. This shows it generates $\textnormal{ind}^G _K \sigma$, and the argument is done.
\end{proof}

\begin{lemma}\label{f_0 +f_1 generate the whole rep}
Assume $\sigma = 1$. If $\overline{f_0 +f_1}$ generates $\textnormal{ind}^G _K 1 /(T+1)$, then $f_0 +f_1$ generates $\textnormal{ind}^G _K 1 $.
\end{lemma}

\begin{proof}
Suppose $f_0 +f_1$ generates $\textnormal{ind}^G _K 1 /(T+1)$. Then we have
\begin{center}
$f_0 \in \langle f_0 +f_1 \rangle_G +(T+1) =\langle f_0 +f_1 \rangle_G +\langle f_0 +f_1 +f_{-1}\rangle_G= \langle f_0 +f_1 \rangle_G +\langle f_0 +f_1 + S_K (f_0 +f_1)\rangle_G = \langle f_0 +f_1 \rangle_G$.
\end{center}
Here the second and the third equality is by Proposition \ref{hecke operator formula} and Proposition \ref{the image of I_1 under S_K and S_-}. The assertion follows.
\end{proof}

\begin{remark}
The converse of the above two Lemmas are certainly true. However, the function $f_1$ generates the universal quotient $\textnormal{ind}^G _K 1/(T+1)$, even it does not generate $\textnormal{ind}^G _K 1$. Explicitly, we have
\begin{center}
$\overline{f_0}= -\overline{f_1 +S_K f_1}$
\end{center}
Similarly, the function $f_0 +f_1$ does not generate $\textnormal{ind}^G _K st$ but it does generate the universal quotient $\textnormal{ind}^G _K st/(T)$.

\end{remark}

\begin{proposition}
The function $f_1$ does not generate $\textnormal{ind}^G _K st /(T)$.
\end{proposition}

\begin{proof}
By Proposition \ref{change of weight}, we have a \emph{non-zero} $G$-map
\begin{center}
$S_{-f_{-1}}: \textnormal{ind}^G _K st \rightarrow \textnormal{ind}^G _K 1 /(T+1)$
\end{center}
sending $f_0$ to $\overline{-f_{-1}}$, which is just $\overline{f_0 +f_1}$ by Proposition \ref{hecke operator formula}. By Proposition \ref{the image of I_1 under S_K and S_-}, the function $f_1$, which equals $S_- f_0$, is sent to $S_- (\overline{f_0 +f_1}) = \overline{f_1 + (-f_1)}=0$, i.e., the image of $f_1$ in the quotient $\textnormal{ind}^G _K 1 /(T+1)$ is zero.  This implies that $f_1$ does not generate $\textnormal{ind}^G _K st$. By Lemma \ref{f_1 generate the whole rep} $\overline{f_1}$ does not generate $\textnormal{ind}^G _K st /(T)$ either. Note that the image of the function $f_1$ in $\textnormal{ind}^G _K st /(T)$ is non-zero.
\end{proof}

\begin{proposition}
The function $f_0 +f_1$ does not generate $\textnormal{ind}^G _K 1 /(T +1)$.
\end{proposition}

\begin{proof}
By Proposition \ref{basis from the trivial weight to steinberg weight}, we have a \emph{non-zero} map
\begin{center}
$\textnormal{ind}^G _K 1 \rightarrow \textnormal{ind}^G _K st /(T)$
\end{center}
sending $f_0$ to $\overline{f_1}$, induced from the map sending $f_0$ to $f_1 +f_{-1}$. One then checks that the map sends $f_0 +f_1$, which equals $f_0 + S_- f_0$, to $\overline{f_1 + S_- f_1}= \overline {f_1 + (-f_1)}= 0$. This proves that $f_0 +f_1$ does not generate $\textnormal{ind}^G _K 1$, and by Lemma \ref{f_0 +f_1 generate the whole rep} it does not generate $\textnormal{ind}^G _K 1/(T+1)$ either. Note that the image of the function $f_0 +f_1$ in $\textnormal{ind}^G _K 1 /(T +1)$ is non-zero.
\end{proof}

\begin{corollary}\label{reducibility of supersingular quotient}
Both $\textnormal{ind}^G _K st /(T)$ and $\textnormal{ind}^G _K 1 /(T +1)$ are not irreducible.
\end{corollary}

\begin{proof}
The assertion follows from last two Propositions.
\end{proof}

\smallskip

\subsection{The regular case that $K=K_1$ and $q=p$}

In the case that $K=K_1$ and $q=p$, we show that $\textnormal{ind}^G_K \sigma /(T_\sigma)$ contains the weight $\sigma^s$. This is analogous to $GL_2 (F)$ for a totally ramified extension $F/\mathbf{Q}_p$ (\cite{Sch14}).  Assume $\sigma$ is a weight of $K$ satisfying that $\chi_\sigma \neq \chi^s _\sigma$.

\begin{lemma}\label{a basis of Kf_1}
The representation $\langle K\cdot f_1\rangle$ has a linear basis given by
\begin{center}
$\{f_1, u\beta_K f_1 \mid u \in N_{n_K}/N_{n_K +1}\}$
\end{center}
\end{lemma}
\begin{proof}
This is by direct computation:

1) first, the Iwahori subgroup $I_K$ acts as $\chi^s _\sigma$ on the function $f_1$. A set of representatives for $K/ I_{K}$ is given by
$\{Id, u\beta_K \mid, u  \in N_{n_K}/N_{n_K +1}\}$. We conclude the representation is spanned by the set $\{f_1, u\beta_K f_1 \mid u \in N_{n_K}/N_{n_K +1}\}$.

2) second, we observe that the functions in the given set have disjoint supports. Recall that
\begin{center}
$\beta_K f_1 =\sum_{u\in N_{n_K +1} /N_{n_K +2}} [u\alpha^{-1}, v_0]$,
\end{center}
then the assertion is clear.
\end{proof}

\begin{corollary}\label{the K-rep generated by f_1}
The representation $\langle K \cdot f_1\rangle$ is isomorphic to $\textnormal{Ind}^{\Gamma_K} _\mathbb{B} \chi^s _\sigma$.
\end{corollary}

\begin{proof}
By Frobenius reciprocity, there is a surjective $K$-map from $\textnormal{Ind}^{\Gamma_K} _\mathbb{B} \chi^s _\sigma$ to $\langle K\cdot f_1\rangle$, sending the function $\varphi_{\chi^s _\sigma}$ to $f_1$. By Lemma \ref{a basis of Kf_1} the representation $\langle K\cdot f_1\rangle$ has the same dimension as that of $\textnormal{Ind}^{\Gamma_K} _\mathbb{B} \chi^s _\sigma$, then the assertion in the statement follows.
\end{proof}

\begin{remark}
We note that Corollary \ref{the K-rep generated by f_1} and Lemma \ref{a basis of Kf_1} does not depend on the assumption of this part.
\end{remark}

\begin{remark}
Using Corollary \ref{the K-rep generated by f_1}, we can see that the $K$-socle of $\textnormal{Ind}^{\Gamma_K} _\mathbb{B} \chi^s _\sigma$ is isomorphic to weight $\sigma$, given by the representation generated by $f_{-1}$.
\end{remark}

\begin{lemma}\label{finite principal series of length 2}
The finite principal series $\textnormal{Ind}^{\Gamma_K} _\mathbb{B} \chi_\sigma$ is of length two, with a unique quotient $\sigma$ and a unique subrepresentation $\sigma^s$, i.e., there is a non-split short exact sequence of $K$-representations:
\begin{center}
$0\rightarrow \sigma^s \rightarrow \textnormal{Ind}^{\Gamma_K} _\mathbb{B} \chi_\sigma \rightarrow \sigma \rightarrow 0$.
\end{center}
\end{lemma}

\begin{proof}
This is \cite[Lemma 5.8]{Karol-Peng2012}. We note this crucially depends on the assumption that $K=K_1$ and $q=p$.
\end{proof}

\begin{proposition}
The universal quotient $\textnormal{ind}^G_K \sigma /(T)$ contains the weight $\sigma^s$.
\end{proposition}

\begin{proof}
The representation $\langle K\cdot \overline{f_1}\rangle$ is by definition equal to $\langle K \cdot f_1\rangle /(T)\cap \langle K \cdot f_1\rangle$. We note that $(T)\cap \langle K \cdot f_1\rangle$ is simply $T([Id, \sigma])$ whence isomorphic to $\sigma$ (explicitly it is the $K$-representation generated by the function $f_{-1}= S_K f_1$). Thus, we conclude by Lemma \ref{finite principal series of length 2} and Corollary \ref{the K-rep generated by f_1} that $\langle K\cdot \overline{f_1}\rangle \cong \sigma^s$.
\end{proof}

\begin{corollary}\label{reducibility of supersingular quotient: the case K=K_1 and q=p}
The representation $\textnormal{ind}^G _K \sigma /(T)$ is not irreducible.
\end{corollary}

\begin{proof}
By Frobenious reciprocity, there is a non-zero $G$-map from the maximal compact induction $\text{ind}^G _K \sigma^s$ to $\text{ind}^G _K \sigma /(T)$, sending the function $f_{0, \sigma^s}$ to $\overline{f_{1, \sigma}}$. Then we see the map sends the function $f_{1, \sigma^s}= S_- f_{0, \sigma^s}$ to $\overline{S_- f_{1, \sigma}}= \overline{c\cdot f_{1, \sigma}}$. However the constant $c$ is zero due to the assumption $\chi^s _\sigma \neq\chi_\sigma$. We conclude that (exchanging $\sigma$ with $\sigma^s$) the function $f_{1, \sigma}$ does not generate $\textnormal{ind}^G _K \sigma$. By Lemma \ref{f_1 generate the whole rep}, the assertion follows.
\end{proof}

\begin{remark}
 We deduce a non-zero $G$-map from $\textnormal{ind}^G _K \sigma /(T)$ to $\textnormal{ind}^G _K \sigma^s /(T)$ and an analogous map of the other direction. However, unlike what happens in the case of $GL_2 (F)$ for a totally ramified extension $F/\mathbf{Q}_p$(\cite[Corollary 2.17]{Sch14}), these maps are not isomorphisms. Indeed, one may check the composition of these maps is zero.
\end{remark}

\begin{remark}
We expect the main result still holds without any restriction on the group $K$ and the size of the residue field $k_F$. However, the strategy we employ seems insufficient to handle the general case not covered by our main result, that is for a weight $\sigma$ of $K$ such that $\chi^s _\sigma \neq \chi_\sigma$, and $(K, q)$ is away from the situation of Corollary \ref{reducibility of supersingular quotient: the case K=K_1 and q=p}. In this case, one might hope to give a straightforward proof that the function $f_1$ does not generate the whole representation $\textnormal{ind}^G _K \sigma$, then one can conclude by Lemma \ref{f_1 generate the whole rep}. One implicit subtlety here is the function $f_1$ would not vanish in any non-supersingular quotient of $\textnormal{ind}^G _K \sigma$, but we have very limited ways to construct a \textbf{non-zero} quotient of the universal supersingular quotient.
\end{remark}

\section{Appendix: the $I_{1, K}$-invariant maps $S_K$ and $S_-$}\label{subsec: S_K and S_-}\label{sec: appendix}
\emph{This part is reproduced from \cite{X2017}.}

\subsection{Definition of $S_K$ and $S_-$}
\emph{In this section, we study some partial linear operators on a smooth representation $\pi$, especially about their certain invariant properties.}

\begin{definition}
Let $\pi$ be a smooth representation of $G$. We define:
\begin{center}
$S_K: \pi^{N'_{m_K}} \rightarrow \pi^{N_{n_K}}$,

$v \mapsto \sum_{u \in N_{n_K}/N_{n_K +1}}~u \beta_K v$.
\end{center}

\begin{center}
$S_-: \pi^{N_{n_K}} \rightarrow \pi^{N'_{m_K}}$,

$v\mapsto \sum_{u'\in N' _{m_K}/N' _{m_K +1}}~u'\beta_K\alpha^{-1}v$
\end{center}

\end{definition}

It is simple to check both $S_K$ and $S_-$ are well-defined.  We summarize the main properties of $S_K$ and $S_-$ as follows:
\begin{proposition}\label{S_K and S_- preserve I_1}
We have:

$(1)$.~~Let $h\in H_1 =I_{1, K}\cap H$. Then $S_K (hv)= h^s\cdot S_K v$, for $v\in \pi^{N'_{m_K}}$, and $S_- (hv)= h^s\cdot S_- v$, , for $v\in \pi^{N_{n_K}}$, where $h^s$ is short for $\beta_K h \beta_K$.

$(2)$.~~If $v$ is fixed by $I_{1, K}$, the same is true for $S_K \cdot v$ and $S_- \cdot v$.
\end{proposition}

\begin{proof}
For $(1)$, we note that the group $H_1$ acts on $\pi^{N_{n_K}}$ and $\pi^{N'_{m_K}}$, as it normalizes $N_{n_K}$ and $N'_{m_K}$. The statement then follows from the definitions.

For $(2)$, we need some preparation, and we sort them out as two lemmas:

\begin{lemma}\label{exchange lemma}
For $u'\in N'_{m_K}, u\in N_{n_K}$, we have:

$(1)$. ~The following identity
\begin{center}
$u'u= u_1 h u'_1$
\end{center}
holds for a unique $u_1\in N_{n_K}, h\in H_1, u'_1 \in N'_{m_K}$.

$(2)$. When $u$ goes through $N_{n_K +l} / N_{n_K +m}$, the element $u_1$ also goes through $N_{n_K +l}/ N_{n_K +m}$, for any $m > l\geq 0$.
\end{lemma}

\begin{proof}
The uniqueness statement is clear, and only the existence needs to be proved.

Assume $u =n(x_1, y_1) \in N, u' \in n'(x, y) \in N'$. Then, if $1+x x_1 + \overline{y y_1}\in E^\times$, we have
 \begin{center}
 $u'u= u_1 h u'_1$,
\end{center}
where $h u'_1$ is the following lower triangular matrix:
\begin{center}
$\begin{pmatrix}  \frac{1}{1+x x_1 + \overline{y y_1}} & 0 & 0  \\ \frac{x-\overline{x_1 y}}{1+\overline{x x_1} +y y_1 } & \frac{1+x x_1+ \overline{y y_1}}{1+ \overline{x x_1} +y y_1} & 0\\
y & yx_1-\bar{x} & 1+ \overline{x x_1}+ y y_1
\end{pmatrix},$
\end{center}
and $u_1= n(x_2, y_2) \in N$, in which $x_2, y_2$ are given by:
\begin{center}
$x_2= \frac{x_1- \overline{y_1 x}}{1+x x_1+\overline{y y_1}}, y_2= \frac{y_1}{1+\overline{x x_1}+y y_1}.$
\end{center}

Under our assumption that $u'\in N'_{m_K}$ and $u\in N_{n_K}$, the condition $1+x x_1 + \overline{y y_1}\in 1+ \mathfrak{p}_E \subset E^\times$ holds automatically. The existence is established.

\medskip
We continue to prove $(2)$. We start by the following observation: from the formula of $y_2$ given in the argument of $(1)$, we see
\begin{center}
$y_2 =y_1 +$ higher valuation terms,
\end{center}
as $u=n(x_1, y_1)\in N_{n_K}, u'= n'(x, y)\in N'_{m_K}$. That is to say:
 \begin{center}
 $u\in N_{n_K +m} \Leftrightarrow u_1\in N_{n_K +m}, \forall m\geq 0$.
 \end{center}

Assume now for another $w\in N_{n_K}$, we have a decomposition $u'w= u_2 b''$ for $u_2\in N_{n_K}$ and $b''\in B'$. We have to prove:
\begin{center}
 $u_2\in u_1 N_{n_K +m}$ implies $w\in u N_{n_K +m}$.
\end{center}
 Write $u^{-1}_1 u_2$ as $u_3$. A little algebraic transform gives:
\begin{center}
$w= u \cdot b'^{-1}u_3 b''$
\end{center}
We need to check that the element $b'^{-1}u_3 b'' \in N_{n_K}$, denoted by $u_4$, lies in $N_{n_K +m}$. The element $b'$ can be written as $h\cdot u'_1$, for a diagonal matrix $h\in H_1$ and $u'_1 \in N'_{m_K}$. We therefore get
\begin{center}
$u'_1 u_4= (h^{-1}u_3 h)\cdot h^{-1}b''$,
\end{center}
where the right hand side is a decomposition of $u'_1 u_4$ given in $(1)$. The uniqueness of such a decomposition implies our observation at the beginning can be applied: we have
$u_4\in N_{n_K +m}$ if and only if $h^{-1}u_3 h\in N_{n_K +m}$ for any $m\geq 0$. Our assumption is that $u_3= u^{-1}_1 u_2\in N_{n_K +m}$, which is the same as $h^{-1}u_3 h\in N_{n_K +m}$ ($h\in H_1$). We are done.
\end{proof}

\begin{remark}\label{refinement of exchange lemme}
A slight variant of $(2)$ holds by the same argument. When $u$ goes through $N_{n_K +l}\setminus N_{n_K +n} / N_{n_K +m}$, the element $u_1$ also goes through $N_{n_K +l}\setminus N_{n_K +n}/ N_{n_K +m}$, for any $m \geq n > l\geq 0$.
\end{remark}
\begin{lemma}\label{exchange lemma 2}
For a $u'\in N'_{m_K}, u\in N_{n_K}$, we have

$(1)$.~The following identity
\begin{center}
$u u'= u'_1 h u_1$
\end{center}
holds for a unique $u'_1 \in N'_{m_K}, h\in H_1, u_1\in N_{n_K}$.

$(2)$.~When $u'$ goes through $N'_{m_K+ l}/ N'_{m_K +m}$, the element $u'_1$ also goes through $N'_{m_K +l}/ N'_{m_K +m}$, for any $m >l \geq 1$.
\end{lemma}

\begin{proof}
The argument of Lemma \ref{exchange lemma} can be slightly modified to work for the current case.
\end{proof}

\smallskip
We proceed to complete the argument of $(2)$ of the Proposition.

By $(1)$ and the decomposition of $I_{1, K}= N'_{m_K} \times H_1 \times N_{n_K}$, it suffices to check that, for $u'= n'(x, y)\in N'_{m_K}$, the element $u'\cdot S_K v$
\begin{center}
$u'\cdot S_K v= \sum_{u\in N_{n_K}/N_{n_K +1}}~u'u\beta_K v$
\end{center}
is still equal to $S_K v=\sum_{u\in N_{n_K}/N_{n_K +1}}~u\beta_K v$. By $(1)$ of Lemma \ref{exchange lemma}, the right hand side of above sum is equal to:
\begin{center}
$\sum_{u\in N_{n_K}/N_{n_K +1}}~u_1 h u'_1 \beta_K v$.
\end{center}
 We get:
\begin{center}
$u'\cdot S_K v=\sum_{u\in N_{n_K}/N_{n_K +1}}~u_1\beta_K (\beta_K h u'_1 \beta_K) v =\sum_{u\in N_{n_K}/N_{n_K +1}}~u_1\beta_K v$,
\end{center}
which, by $(2)$ of Lemma \ref{exchange lemma}, is just $\sum_{u_1\in N_{n_K}/N_{n_K +1}}~u_1\beta_K v$. The argument for the statement $S_K v \in \pi^{I_{1, K}}$ for $v\in \pi^{I_{1, K}}$ is complete now.

\medskip
Using Lemma \ref{exchange lemma 2}, the previous argument can be slightly modified to show that $S_- v \in \pi^{I_{1, K}}$ for $v\in \pi^{I_{1, K}}$. The argument of the Proposition is done.
\end{proof}

\subsection{The images of $(\textnormal{ind}_K ^G \sigma)^{I_{1, K}}$ under $S_K$ and $S_-$}\label{sec: the image of pro-p under two operators}

\begin{proposition}\label{the image of I_1 under S_K and S_-}
We have:

$(1)$.~For $n\geq 1$,
\begin{center}
$S_K f_n = f_{-n}$, ~$S_- f_n = c_- f_n$.
\end{center}
Here, the constant $c_-$ is given by:
\begin{center}
$c_-= \sum_{(x, t)\in L^\times _{q^{4-t_K}}}\chi_\sigma (h(t))$.
\end{center}

$(2)$.~For $n \geq 0$,
\begin{center}
$S_K f_{-n}= d_n f_{-n}$, ~$S_- f_{-n}= f_{n+1}$.
\end{center}
Here, the constant $d_n$ $(n\geq 1)$ is given by:
\begin{center}
$d_n =\sum_{(x, t)\in L^\times _{q^{t_K}}} \chi_\sigma ((h(t))$;
\end{center}
and the constant $d_0$ is equal to:
\begin{center}
$d_0= \begin{cases}
-\chi_\sigma (h(\mathfrak{t})), ~~\text{if}~\sigma\cong~\text{a twist of the Steinberg weight}
;\\
0, ~~~~~~~~~~~~~~~~~~\text{otherwise}.\end{cases}$
\end{center}
\end{proposition}

\begin{proof}
We will prove $S_K f_n = f_{-n}$ for $n\geq 1$ and $S_- f_{-n}= f_{n+1}$ for $n \geq 0$ at first.

For $n\geq 1$, the support of the function $S_K f_n$ is contained in:
\begin{center}
$K\alpha^{-n}I_{1, K}\beta_K N_{n_K} =K\alpha^n I_{1, K}$.
\end{center}
Then, by Proposition \ref{S_K and S_- preserve I_1} and \cite[Remark 3.8]{X2019}, the function $S_K f_n$ is proportional to $f_{-n}$. We compute:
\begin{center}
$S_K f_n (\alpha^n)= \sum_{u \in N_{n_K}/N_{n_K +1}}~f_n(\alpha^n u\beta_K)= f_n (\alpha^n \beta_K)= v_0$,
\end{center}
where we note that $\alpha^n u\beta_K \in K\alpha^n I_{1, K}$, for $u\in N_{n_K}\setminus N_{n_K +1}$ (\cite[(3) of Proposition 6.1]{X2018}). Hence, we have proved $S_K f_n= f_{-n} $ for $n\geq 1$.

\smallskip
For $n\geq 0$, the support of the function $S_- f_{-n}$ is contained in
\begin{center}
$K\alpha^n I_{1, K}\beta_K \alpha^{-1}N'_{m_K}= K\alpha^{-(n+1)}I_{1, K}$.
\end{center}
 By Proposition \ref{S_K and S_- preserve I_1} and \cite[Remark 3.8]{X2019} again, the function $S_- f_{-n}$ is proportional to $f_{n+1}$. We compute:
\begin{center}
$S_- f_{-n} (\alpha^{-(n+1)})= \sum_{u' \in N'_{m_K}/N'_{m_K +1}}~f_{-n} (\alpha^{-(n+1)}u'\alpha\beta_K)= \beta_K v_0$,
\end{center}
where we note that $\alpha^{-(n+1)}u'\alpha\beta_K \in K\alpha^{n+1} K$, for $u'\in N'_{m_K} \setminus N'_{m_K +1}$ (\cite[(3) of Proposition 6.1]{X2018}). Thus, we have verified $S_- f_{-n}= f_{n+1}$, for $n\geq 0$.

\smallskip

We proceed to prove $S_- f_n =c_- f_n$ for $n\geq 1$: we will determine the value of $c_-$ explicitly. The support of the function $S_- f_n$ is contained in
\begin{center}
$K\alpha^{-n}I_{1, K} \alpha \beta_K N'_{m_K} \subseteq K \alpha^{n-1} I_{1, K} \cup K \alpha^n I_{1, K}$,
\end{center}
where the inclusion follows from \cite[(1) and (3) of Proposition 6.1]{X2018}. We conclude that $S_- f_n \in \langle f_{-(n-1)}, f_n\rangle$ by Proposition \ref{S_K and S_- preserve I_1} and \cite[Remark 3.8]{X2019}. We compute:
\begin{center}
$S_- f_n (\alpha^{n-1})=\sum_{u'\in N' _{m_K}/N' _{m_K +1}} f_n (\alpha^{n-1}u'\alpha \beta_K)= \sum_{u'\in N' _{m_K}/N' _{m_K +1}} v_0 =0.$
\end{center}
It remains to compute $S_- f_n (\alpha^{-n})$:
\begin{center}
$S_- f_n (\alpha^{-n})=\sum_{u'\in N' _{m_K}/N' _{m_K +1}} f_n (\alpha^{-n} u'\alpha \beta_K)$.
\end{center}
Note that $\alpha^{-n} u'\alpha \beta_K \in K\alpha^{n-1}I_{1, K}$ for $u'\in N' _{m_K +1}$, and we are reduced to
\begin{center}
$S_- f_n (\alpha^{-n})=\sum_{u'\in (N' _{m_K}\setminus N' _{m_K +1})/N' _{m_K +1}} f_n (\alpha^{-n} u'\alpha \beta_K)$
\end{center}
For a $u' =n'(\ast, \varpi^{m_K} _E t)$ for some $t\in \mathfrak{o}^\times _E$, we have (using \eqref{useful identity})
\begin{center}
$\alpha^{-n} u'\alpha \beta_K =n (\ast, \varpi^{2n-1+n_K} _E t^{-1})h(\bar{t}^{-1})\alpha^{-n}n'(\ast, \varpi^{m_K} _E t^{-1})$.
\end{center}
Thus, we immediately get:
\begin{center}
$S_- f_n (\alpha^{-n})=(\sum_{(x, t)\in L^\times _{q^{4-t_K}}}\chi_\sigma (h(t))) \beta_K v_0$,
\end{center}
here we have identified the group $N' _{m_K}/N' _{m_K +1}$ with $L _{q^{4-t_K}}$, via the map $L_{m_K}$.

Hence, we get
\begin{center}
$c_-= \sum_{(x, t)\in L^\times _{q^{4-t_K}}}\chi_\sigma (h(t))$.
\end{center}

\smallskip
We move to deal with the last statement: $S_K f_{-n}= d_n f_{-n}$, for $n\geq 0$. The support of the function $S_K f_{-n}$ is contained in
\begin{center}
$K\alpha^n I_{1, K}\beta_K N_{n_K} \subseteq K \alpha^n K$
\end{center}
By \cite[Remark 3.8]{X2019}, we get:

when $n=0$, $S_K f_0 \in \langle f_0\rangle$;

when $n >0$, $S_K f_{-n} \in \langle f_{-n}, f_n\rangle$.

We consider the second case at first. Assume $n >0$. We compute:
\begin{center}
$S_K f_{-n} (\alpha^{-n})= \sum_{u \in N_{n_K}/N_{n_K +1}}~f_{-n} (\alpha^{-n}u\beta_K)=\sum_{u \in N_{n_K}/N_{n_K +1}} \beta_K v_0= 0$.
\end{center}

Next, we compute $S_K f_{-n} (\alpha^n)$:
\begin{center}
$S_K f_{-n} (\alpha^n)= \sum_{u \in N_{n_K}/N_{n_K +1}}~f_{-n} (\alpha^n u\beta_K)$.
\end{center}
Note that $\alpha^n u\beta_K \in K\alpha^{-n}N'_{m_K}$, for $u\in N_{n_K +1}$. We are thus reduced to:
\begin{center}
$S_K f_{-n} (\alpha^n)= \sum_{u \in (N_{n_K}\setminus N_{n_K +1})/N_{n_K +1}}~f_{-n} (\alpha^n u\beta_K)$.
\end{center}
For $u= n(\ast, \varpi^{n_K} _E t)$, for some $t\in \mathfrak{o}^\times _E$, we have (using \eqref{useful identity}):
\begin{center}
$\alpha^n u\beta_K= n'(\ast, \varpi^{2n-1+m_K} _E t^{-1})h(t)\alpha^n n(\ast, \varpi^{n_K} _E t^{-1})$.
\end{center}
Thus, we get
\begin{center}
$S_K f_{-n} (\alpha^n)=(\sum_{(x, t)\in L^\times _{q^{t_K}}} \chi_\sigma ((h(t))) v_0$,
\end{center}
here we have identified the group $N_{n_K}/N_{n_K +1}$ with $L _{q^{t_K}}$, via the map $L_{n_K}$. Hence, we get:
\begin{center}
$d_n =\sum_{(x, t)\in L^\times _{q^{t_K}}} \chi_\sigma ((h(t))$
\end{center}

\begin{remark}
The exact values of $c_-$ and $d_n$ ($n\geq 1$) depend on the nature of the character $\chi_\sigma$, and they have been computed explicitly in \cite[Appendix A]{Karol-Peng2012}.
\end{remark}

\medskip
We still need to compute the constant $d_0$ appearing in $S_K f_0 =d_0 f_0$. By definition, the constant $d_0$ is determined by
\begin{equation}\label{equation determines d_0}
\sum_{u \in N_{n_K}/N_{n_K +1}} u\beta_K v_0 =d_0 v_0.
\end{equation}

\medskip
We recall some stuff from \cite[section 5]{Karol-Peng2012}:

$1)$.~(Definition 5.2 of \emph{loc.cit})

To any character $\chi$ of $H_0 /H_1$, a subset $J_K (\chi)\subset \{s\}$ is attached.

$2)$.~(Definition 5.3 of \emph{loc.cit})

For any subset $J \subset J_K (\chi)$, one defines a character $M_{\chi, J}$ of the finite Hecke algebra $\mathcal{H}_{\Gamma_K} := \text{End}_{\Gamma_K} (\text{Ind}^{\Gamma_K} _\mathbb{U} 1)$.

$3)$.~(Proposition 5.4 of \emph{loc.cit})

Every simple module of the algebra $\mathcal{H}_{\Gamma_K}$ is isomorphic to $M_{\chi, J}$ for some character $\chi$ of $H_0 /H_1$ and some $J\subset J_K (\chi)$.

$4)$.~(Proposition 5.5 of \emph{loc.cit})

The functor $\sigma \rightarrow \sigma^{\mathbb{U}}$ gives a bijection between the set of isomorphism classes of irreducible representations of $\Gamma_K$ and the set of isomorphism classes of simple right $\mathcal{H}_{\Gamma_K}$-modules.

\medskip

By $4)$ above, we write our $\sigma$ as $\sigma_{\chi_\sigma, J}$ such that:
\begin{center}
 $\sigma^\mathbb{U} \cong M_{\chi_\sigma, J}$,
 \end{center}
for some $J\subset J_K (\chi_\sigma)$. Then, by comparing \eqref{equation determines d_0} and the right action of $\mathcal{H}_{\Gamma_K}$ on $\sigma^{\mathbb{U}}$(\cite[3.1, (1)]{Karol-Peng2012}), we see immediately that
\begin{center}
$d_0= M_{\chi_\sigma, J} (T_{\beta_K})$,
\end{center}
where $T_{\beta_K}$ is the Hecke operator in $\mathcal{H}_{\Gamma_K}\hookrightarrow \mathcal{H} (I_{1, K}, 1)$ which corresponds to the double coset $I_{1, K}\beta_K I_{1, K}$. By the identification in \cite[Proposition 5.7]{Karol-Peng2012}, our statement for the value of $d_0$ now follows from the lists in Definition 5.3 of \emph{loc.cit}:
\begin{center}
$d_0= \begin{cases}
-\chi_\sigma (h(\mathfrak{t})), ~~\text{if}~\sigma\cong~\text{a twist of}~st
\\
0, ~~~~~~~~~~~~~~~~~~\text{otherwise}\end{cases}$
\end{center}
Here, we note that the element $\beta_K$ is different from a normalized one\footnote{i.e., an element of determinant $1$.} used in \cite{Karol-Peng2012} by exactly the diagonal matrix $h(\mathfrak{t})$.
\end{proof}
\smallskip

\bibliographystyle{amsalpha}
\bibliography{new}

\end{document}